\documentclass[review,11pt]{elsarticle}
\usepackage{amssymb,multirow,mathrsfs,subfigure}
\usepackage{amsmath}
\usepackage{array}
\usepackage{algorithm}
\usepackage{algorithmicx}
\usepackage[noend]{algpseudocode}
\usepackage{paralist}
\usepackage{natbib}
\usepackage{hyperref}
\usepackage{epsfig}
\usepackage{color,soul,cancel}
\numberwithin{equation}{section}
\usepackage{mathtools,lipsum}
\usepackage[toc,page]{}
\graphicspath{{images/}{../images/}}
\usepackage{bigints,amsthm}
\allowdisplaybreaks
\newtheorem{theorem}{Theorem}[section]
\newtheorem{corollary}{Corollary}[section]

\makeatletter
\def\ps@pprintTitle{%
\let\@oddhead\@empty
\let\@evenhead\@empty
\let\@oddfoot\@empty
\let\@evenfoot\@oddfoot
}
\makeatother

\begin{document}
\begin{frontmatter}
\title{The correspondence between Voigt and Reuss bounds and the decoupling constraint in a two-grid staggered solution algorithm to coupled flow and deformation in heterogeneous poroelastic media}
\author[csm]{Saumik Dana}
\ead{saumik@utexas.edu}
\author[ita]{Joel Ita}
\ead{j.ita@shell.com}
\author[csm]{Mary. F. Wheeler}
\ead{mfw@ices.utexas.edu}

\address[csm]{Center for Subsurface Modeling, Institute for Computational Engineering and Sciences, The University of Texas at Austin, TX 78712}
\address[ita]{Shell PetroSigns Development Manager and PTE Geomechanics
3333 Hwy 6 South Houston, TX 77082}
\begin{abstract}
We perform a convergence analysis of a two-grid staggered solution algorithm for the Biot system modeling coupled flow and deformation in heterogeneous poroelastic media. The algorithm first solves the flow subproblem on a fine grid using a mixed finite element method (by freezing a certain measure of the mean stress) followed by the poromechanics subproblem on a coarse grid using a conforming Galerkin method. Restriction operators map the fine scale flow solution to the coarse scale poromechanical grid and prolongation operators map the coarse scale poromechanical solution to the fine scale flow grid. The coupling iterations are repeated until convergence and Backward Euler is employed for time marching. The analysis is based on studying the equations satisfied by the difference of iterates to show that the two-grid scheme is a contraction map under certain conditions. Those conditions are used to construct the restriction and prolongation operators as well as arrive at coarse scale elastic properties in terms of the fine scale data. We show that the adjustable parameter in the measure of the mean stress is linked to the Voigt and Reuss bounds frequently encountered in computational homogenization of multiphase composites.
\end{abstract}
\end{frontmatter}
\section{Introduction}
\begin{figure}[h!]
\centering
\subfigure[]
{\includegraphics[scale=1]{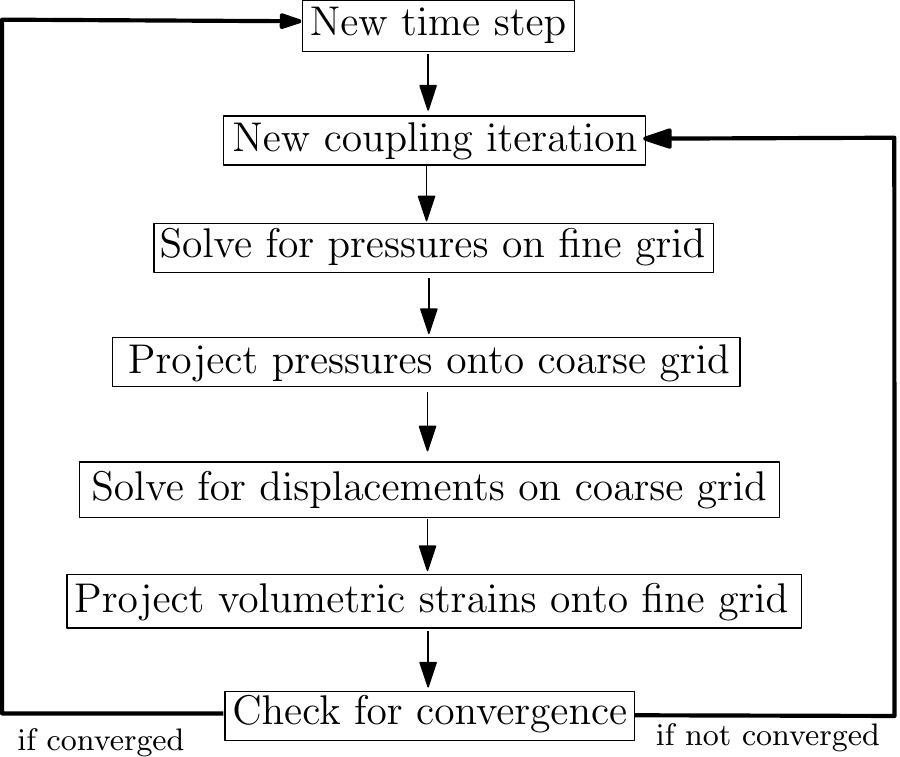}\label{figa}}
\hspace{25pt}
\subfigure[]{\includegraphics[scale=0.55]{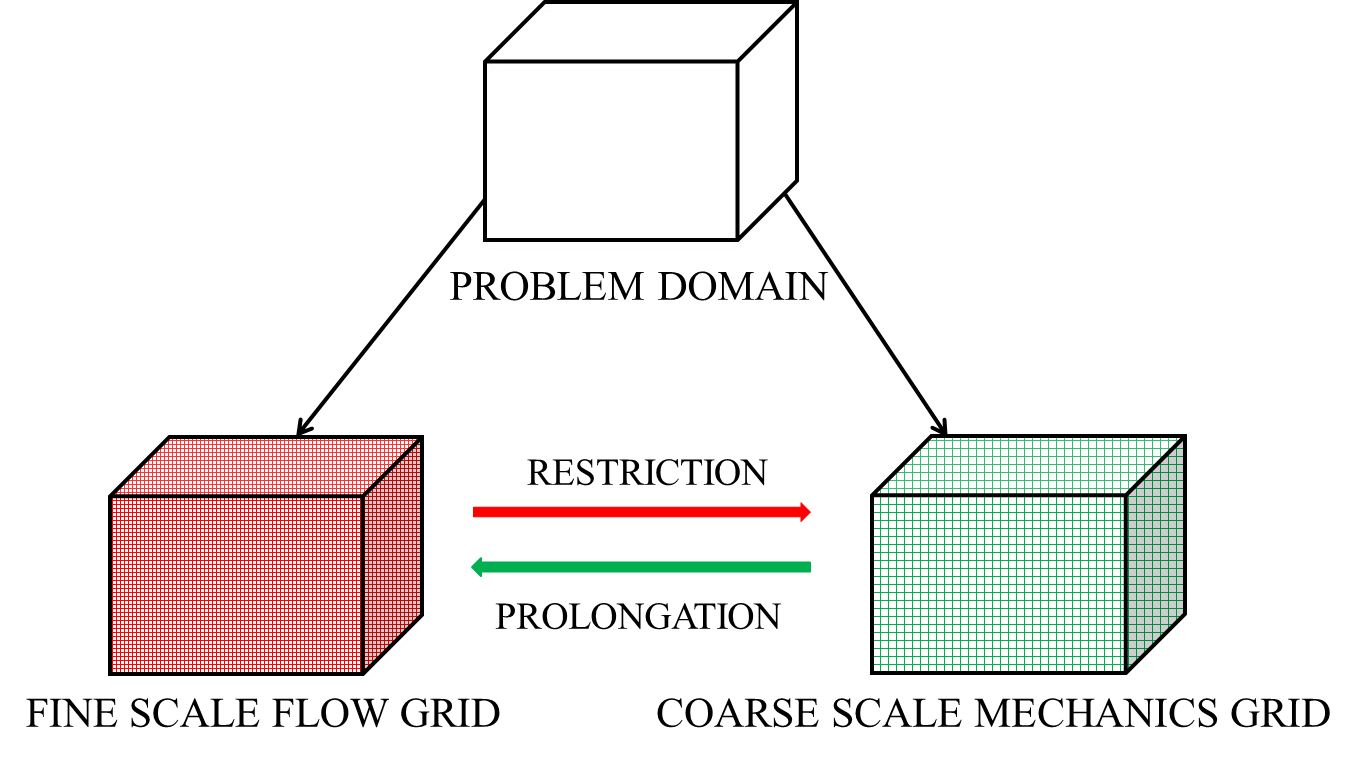}\label{figb}}
\caption{\ref{figa}: Two-grid staggered solution algorithm. A measure of the mean stress remains fixed during the flow solve. After the flow solve, the updated pressures are projected onto the coarse scale poromechanics grid. After the poromechanics solve, the updated volumetric strains are projected onto the fine scale flow grid. \ref{figb}: One coupling iteration of two-grid scheme. In order to be consistent with the terminology used in multigrid methods, we refer to projection onto coarse grid as `restriction' and projection onto fine grid as `prolongation'.}
\label{break}
\end{figure} 
Staggered solution algorithms are used to decompose coupled problems into subproblems which are then solved sequentially in successive iterations until a convergence criterion is met at each time step (\citet{Carlo}, \citet{armero}, \citet{turksa}, \citet{schrefler}). These algorithms offer avenues for augmentations in which subproblems associated with fine scale phenomena can be solved on a fine grid and subproblems associated with coarse scale phenomena can be solved on a coarse grid. Consolidation in deep subsurface reservoirs has inherent length scale disparities with fine scale features of multiphase flow restricted to the reservoir and coarse scale features of geomechanical deformation associated with a domain including but not restricted to the reservoir. In lieu of the above, \citet{dana-2018} developed a two-grid staggered solution algorithm in which the flow equations are solved on a fine grid and the poromechanics equations are solved on a coarse grid (with the grids being non-nested) in every coupling iteration in every time step and used the classical Mandel's problem (\citet{mandel-1953}, \citet{ref27}) to show that the scheme is numerically convergent. Thereafter, motivated by the previous work of \citet{ref12} and \citet{almani-cg}, \citet{danacmame} established theoretical convergence of the two-grid scheme of \citet{dana-2018} for the degenerate case of nested brick grids with the flow and poromechanical domains being identical, as shown in Figure \ref{break}. The measure of mean stress that remains fixed during the flow solve is hydrostatic part of the total stress, also refered to as the mean stress. The interesting result of the work of \citet{danacmame} is that the convergence analysis lends itself to an expression for coarse scale bulk moduli in terms of fine scale bulk moduli, and further the coarse scale moduli are a harmonic mean of the fine scale moduli. The harmonic mean is exactly the Reuss bound (see \citet{saeb-2016}). This observation leads to a hypothesis that there must be a measure of mean stress which when fixed during the flow solve in the two-grid approach, leads to the arithmetic mean (Voigt bound) for coarse scale bulk moduli in terms of fine scale bulk moduli. We already know that the Reuss and Voigt bounds on effective moduli yield the lower and upper bounds for the elastic strain energy for multiphase composites respectively (see \citet{saeb-2016}). The objective of this work is to examine the link between the decoupling constraint used in the two-grid approach and effective coarse scale property that the convergence analysis lends itself to. With that in mind, we define a measure of mean stress which equates to the actual mean stress only as a special case. As a result, the staggering in this work is a generalization of the fixed stress split staggering that was studied in \citet{ref12}, \citet{almani-cg} and \citet{danacmame}. This paper is structured as follows: Section 2 presents the model equations for flow and poromechanics, Section 3 presents the statement of contraction of the two-grid fixed stress split iterative scheme, Section 4 presents the details of how the statement of contraction is used to arrive at restriction and prolongation operators as well as the effective coarse scale moduli, Section 5 presents the two-grid fixed stress split algorithm and Section 6 discusses the link between the decoupling constraint and the Voigt and Reuss bounds.
\subsection{Preliminaries}
Given a bounded convex domain $\Omega\subset \mathbb{R}^3$, we use $Meas(\Omega)$ to denote the volume of $\Omega$, $\mathbb{P}_k(\Omega)$ to represent the restriction of the space of polynomials of degree less that or equal to $k$ to $\Omega$ and $\mathbb{Q}_1(\Omega)$ to denote the space of trilinears on $\Omega$. For the sake of convenience, we discard the differential in the integration of any scalar field $\chi$ over $\Omega$ as follows
\begin{align*}
\tag{$\forall \mathbf{x}\in \Omega$}
\int\limits_{\Omega}\chi(\mathbf{x}) \equiv \int\limits_{\Omega}\chi(\mathbf{x}) \,dV
\end{align*}
Sobolev spaces are based on the space of square integrable functions on $\Omega$ given by
\begin{align*}
L^2(\Omega)\equiv \big\{\theta:\Vert \theta\Vert_{\Omega}^2:=\int\limits_{\Omega}\vert\theta\vert^2 < +\infty \big\},
\end{align*}
\section{Model equations}
\subsection{Flow model}\label{flowmodel1} 
The fluid mass conservation equation \eqref{massagain1} in the presence of deformable porous medium with the Darcy law \eqref{darcyagain1} and linear pressure dependence of density \eqref{compressible1} with boundary conditions \eqref{bcagain11} and initial conditions \eqref{flowendagain1} is
\begin{align}
\label{massagain1}
&\frac{\partial \zeta}{\partial t}+\nabla \cdot \mathbf{z}=q\\
\label{darcyagain1}
&\mathbf{z}=-\frac{\mathbf{K}}{\mu}(\nabla p-\rho_0 \mathbf{g})=-\boldsymbol{\kappa}(\nabla p-\rho_0 \mathbf{g})\\
\label{compressible1}
&\rho=\rho_0(1+c\,(p-p_0))\\
\label{bcagain11}
&p=g \,\, \mathrm{on}\,\,\Gamma_D^f \times (0,T],\,\,\mathbf{z}\cdot\mathbf{n}=0 \,\, \mathrm{on}\,\,\Gamma_N^f \times (0,T]
\\
\label{flowendagain1}
&p(\mathbf{x},0)=p_0(\mathbf{x}),\,\,\rho(\mathbf{x},0)=\rho_0(\mathbf{x}),\,\, \phi(\mathbf{x},0)=\phi_0(\mathbf{x})\qquad (\forall \mathbf{x}\in \Omega)
\end{align}
where $p:\Omega \times (0,T]\rightarrow \mathbb{R}$ is the fluid pressure, $\mathbf{z}:\Omega \times (0,T]\rightarrow \mathbb{R}^3$ is the fluid flux, $\bar{\epsilon}$ is the volumetric strain, $\Gamma_D^f$ is the Dirichlet boundary, $\mathbf{n}$ is the unit outward normal on the Neumann boundary $\Gamma_N^f$, $q$ is the source or sink term, $\mathbf{K}$ is the uniformly symmetric positive definite absolute permeability tensor, $\mu$ is the fluid viscosity, $\rho_0$ is a reference density, $\boldsymbol{\kappa}=\frac{\mathbf{K}}{\mu}$ is a measure of the hydraulic conductivity of the pore fluid, $c$ is the fluid compressibility, $T>0$ is the time interval, $\zeta\equiv \frac{1}{M}p+\alpha \bar{\epsilon}$ is refered to as the fluid content (see \citet{biot3}, \citet{rice}, \citet{review}, \citet{coussy}) where $\alpha\equiv 1-\frac{K_b}{K_s}$ is the Biot constant (see \citet{biot1}, \citet{geertsma}, \citet{nur}) and $M\equiv \frac{1}{\phi_0 c+\frac{(\alpha-\phi_0)(1-\alpha)}{K_b}}$ is the Biot modulus (see \citet{biot3}) with $K_b$ being the drained bulk modulus of the pore skeleton and $K_s$ being the bulk modulus of the solid grains. For the sake of convenience, we introduce a variable $\varphi\equiv \frac{1}{M}+\frac{\alpha^2}{\eta}$, where $\eta$ is an adjustable parameter as we shall in Module \ref{bambam}.
\subsection{Poromechanics model}\label{poromodel1}
The linear momentum balance \eqref{mechstart} in the quasi-static limit of interest with the definition of the total stress \eqref{effective} (see \citet{biot1}) with the expression for the body force \eqref{force} and the small strain assumption \eqref{smallstrain} with boundary conditions \eqref{bc1} and initial condition \eqref{initporo} is
\begin{align}
\label{mechstart}
&\nabla\cdot \boldsymbol{\sigma}+\mathbf{f}=\mathbf{0}\\
\label{effective}
&\boldsymbol{\sigma}=\boldsymbol{\sigma}_0
+\lambda \bar{\epsilon}\mathbf{I}+2G\boldsymbol{\epsilon}
-\alpha (p-p_0)\mathbf{I}\\
\label{force}
&\mathbf{f}=\rho \phi\mathbf{g} + 
\rho_r(1-\phi)\mathbf{g}\\
\label{smallstrain}
&\boldsymbol{\epsilon}(\mathbf{u})
=\frac{1}{2}(\nabla \mathbf{u} + \nabla^T \mathbf{u})\\
\label{bc1}
&\mathbf{u}\cdot\mathbf{n}_1=0\,\, \mathrm{on}\,\,\Gamma_D^p \times [0,T],\,\,\boldsymbol{\sigma}^T\mathbf{n}_2=\mathbf{t}\,\,\mathrm{on}\,\,\Gamma_N^p \times [0,T]\\
\label{initporo}
&\mathbf{u}(\mathbf{x},0)=\mathbf{0}\qquad (\forall\,\,\mathbf{x}\in \Omega)
\end{align}
where $\mathbf{u}:\Omega \times [0,T]\rightarrow \mathbb{R}^3$ is the solid displacement, $\rho_r$ is the rock density, $G$ is the shear modulus, $\nu$ is the Poisson's ratio, $\mathbf{n}_1$ is the unit outward normal to the Dirichlet boundary $\Gamma_D^p$, $\mathbf{n}_2$ is the unit outward normal to the Neumann boundary $\Gamma_N^p$, $\alpha$ is the Biot parameter, $\mathbf{f}$ is body force per unit volume, $\mathbf{t}$ is the traction boundary condition, $\boldsymbol{\epsilon}$ is the strain tensor, $\bar{\epsilon}$ is the volumetric strain, $\boldsymbol{\sigma}_0$ is the in situ stress, $\lambda$ is the Lame parameter and $\mathbf{I}$ the is second order identity tensor.
\subsection{The decoupling assumption}\label{bambam}
The basic idea of the two-grid staggered solution strategy is to solve the flow system \eqref{massagain1}-\eqref{flowendagain1} on a fine grid for the pressures at the current coupling iteration based on the value of a certain measure of mean stress from the previous coupling iteration. We refer to that measure of mean stress as $\bar{\sigma}$, and is expressed as follows
\begin{align*}
\bar{\sigma}=\eta \bar{\epsilon}-\alpha p
\end{align*}
where $\eta$ is an adjustable parameter, which when equated to the drained bulk modulus, lends itself to the total mean stress (refered to as $\sigma_v$) as follows
\begin{align*}
\bar{\sigma}=K_b\bar{\epsilon}-\alpha p\equiv \sigma_v\qquad (\mathrm{when}\,\,\eta=K_b)
\end{align*} 
These pressures are then fed to the poromechanics system \eqref{mechstart}-\eqref{initporo} which is solved for displacements on a coarse grid thereby updating the stress state. This updated stress state is then fed back to the flow system for the next coupling iteration. Since this strategy condemns the porous solid to follow a certain stress path during the flow solve, the convergence of the solution algorithm is not automatically guaranteed. It is important to note that the adjustable $\eta
$ allows for flexibility in the choice of decoupling constraint, and the fixed stress split strategy is only a special case when the adjustable parameter is identical to the drained bulk modulus i.e. when $\eta=K_b$. 
\section{Statement of contraction of the two-grid fixed stress split scheme}
\begin{figure}
\hspace{-40pt}
\includegraphics[scale=1.2]{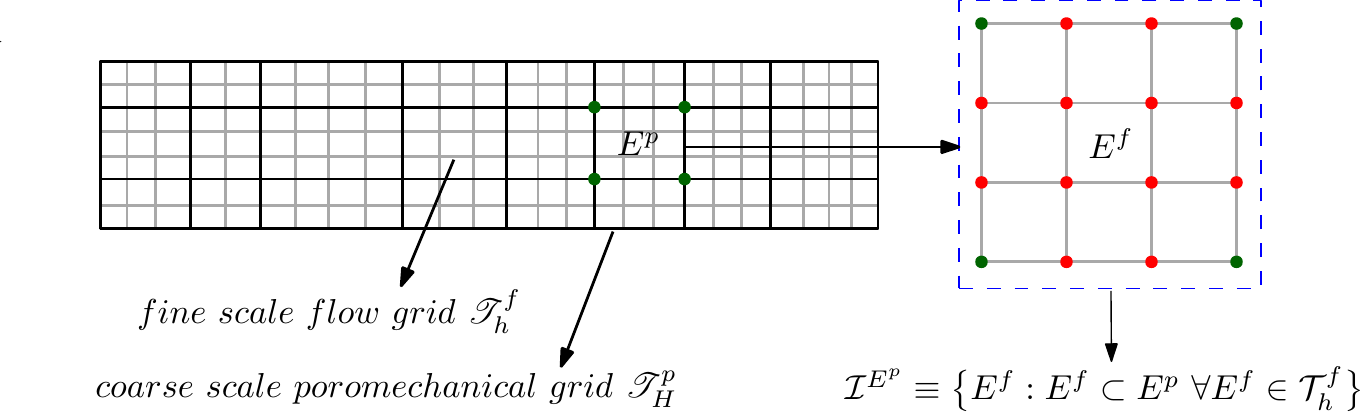}
\caption{Depiction of nested grids (in a two-dimensional framework for the sake of convenience). Red dots are vertices of flow element(s) and green dots are vertices of poromechanical element(s).}
\label{nested}
\end{figure}
The objective of our analysis is to arrive at a contraction map for the fully discrete two-grid staggered solution algorithm while taking into account the heterogeneities in the underlying porous medium. Let $\mathscr{T}_h^f$ represent the fine scale flow grid consisting of brick elements $E^f$ and $\mathscr{T}_H^p$ be the coarse scale poromechanical grid consisting of brick elements $E^p$ such that 
\begin{align*}
r=\frac{\max\limits_{E^p\in\mathscr{T}_H^p}diam(E^p)}
{\max\limits_{E^f\in\mathscr{T}_h^f}diam(E^f)}\geq 1
\end{align*}
Since the grids are nested, each coarse scale poromechanical element $E^p\in \mathscr{T}_H^p$ can be viewed as a union of flow elements belonging to the set $\mathcal{I}^{E^p}$ as follows
\begin{align*}
&E^p=\bigcup\limits_{E^f\in \mathcal{I}^{E^p}} E^f\qquad \mathrm{where}\qquad \mathcal{I}^{E^p}\equiv \big\{E^f:E^f\subset E^p\,\,\forall E^f\in \mathcal{T}_{h}^f\big\}
\end{align*}
To take into account the underlying heterogeneities in the porous medium, we introduce the notations $(\cdot)_{E^f}$ for the value of any material parameter $(\cdot)$ at flow element $E^f$ and $(\cdot)_{E^p}$ for the value of any material parameter $(\cdot)$ evaluated at poromechanics element $E^p$.
\subsection{Variational statements in terms of coupling iteration differences} 
We use the notations $(\cdot)^{n+1}$ for any quantity $(\cdot)$ evaluated at time level $n+1$, $(\cdot)^{m,n+1}$ for any quantity $(\cdot)$ evaluated at the $m^{th}$ coupling iteration at time level $n+1$, $\delta^{(m)}_f (\cdot)$ for the change in the quantity $(\cdot)$ during the flow solve in the $(m+1)^{th}$ coupling iteration at any time level and $\delta^{(m)} (\cdot)$ for the change in the quantity $(\cdot)$ over the $(m+1)^{th}$ coupling iteration at any time level. The discrete variational statements in terms of coupling iteration differences is : find $\delta^{(m)} p_h\in W_h$, $\delta^{(m)} \mathbf{z}_h\in \mathbf{V}_h$ and $\delta^{(m)} \mathbf{u}_H\in \mathbf{U}_H$ such that
\begin{align}
\label{msone}
&\sum\limits_{E^f\in \mathscr{T}_h^f}\varphi_{E^f}(\delta^{(m)} p_h,\theta_h)_{E^f}+\sum\limits_{E^f\in \mathscr{T}_h^f}\Delta t(\nabla \cdot \delta^{(m)} \mathbf{z}_h,\theta_h)_{E^f}=-\sum\limits_{E^f\in \mathscr{T}_h^f}\frac{\alpha_{E^f}}{\eta_{E^f}}(\delta^{(m-1)} \bar{\sigma},\theta_h)_{E^f}\\
\label{mstwo}
&\sum\limits_{E^f\in \mathscr{T}_h^f}
(\boldsymbol{\kappa}_{E^f}^{-1}\delta^{(m)} \mathbf{z}_h, \mathbf{v}_h)_{E^f}=\sum\limits_{E^f\in \mathscr{T}_h^f}(\delta^{(m)} p_h,\nabla \cdot \mathbf{v}_h)_{E^f}\\
\label{msthree}
&\sum\limits_{E^p\in \mathscr{T}_H^p} 2G_{E^p}(\boldsymbol{e}(\delta^{(m)} \mathbf{u}_H),\boldsymbol{e}(\mathbf{q}_H))_{E^p}+\sum\limits_{E^p\in \mathscr{T}_H^p}(\delta^{(m)} \bar{\sigma},\nabla \cdot \mathbf{q}_H)_{E^p}=0
\end{align}
where the finite dimensional spaces $W_h$, $W_H$ $\mathbf{V}_h$ and $\mathbf{U}_H$ are given by 
\begin{align*}
&W_h= \big\{\theta_h:\theta_h\vert_{E^f}\in \mathbb{P}_0(E^f)\,\,\forall E^f\in \mathscr{T}_h^f\big\}\\
&W_H= \big\{\theta'_H:\theta'_H\vert_{E^p}\in \mathbb{P}_0(E^p)\,\,\forall E^p\in \mathscr{T}_H^p\big\}\\
&\mathbf{V}_h=\big\{\mathbf{v}_h:\mathbf{v}_h\vert_{E^f}\leftrightarrow \hat{\mathbf{v}}\vert_{\hat{E}}:\hat{\mathbf{v}}\vert_{\hat{E}}\in \hat{\mathbf{V}}(\hat{E})\,\,\forall E^f\in \mathscr{T}_h^f,\,\,\mathbf{v}_h \cdot \mathbf{n}=0\,\,\mathrm{on}\,\,\Gamma_N^f\big\}\\
&\mathbf{U}_H=\big\{\mathbf{q}_H=(u,v,w):u\vert_{E^p},
v\vert_{E^p},w\vert_{E^p}\in \mathbb{Q}_1(E^p)\,\,\forall E^p\in \mathscr{T}_H^p,\mathbf{q}_H=\mathbf{0}\,\,\mathrm{on}\,\,\Gamma_D^p\big\}
\end{align*}   
and the details of $\hat{\mathbf{V}}(\hat{E})$ are given in \citet{dana-2018}. The equations \eqref{msone}, \eqref{mstwo} and \eqref{msthree} are the discrete variational statements (in terms of coupling iteration differences) of \eqref{massagain1}, \eqref{darcyagain1} and \eqref{mechstart} respectively. The details of \eqref{msone} and \eqref{mstwo} are given in \ref{discreteflow} whereas the details of \eqref{msthree} are given in \ref{discretemechanics}.
\subsection{Restriction and prolongation operators}
We introduce the restriction operator $\mathscr{R}$ that maps the fine scale pressure solution onto the coarse scale poromechanics grid and the prolongation operator $\mathscr{P}$ that maps the coarse scale volumetric strain onto the fine scale flow grid as follows
\begin{align*}
&\mathscr{R}:W_h\mapsto W_H\\
&\mathscr{P}:\nabla \cdot \mathbf{U}_H\mapsto W_h
\end{align*}  
As a result, the measure of the mean stress is defined on the fine and coarse grids as
\begin{align}
\label{use1}
\bar{\sigma}=\eta_{E^p}\bar{\epsilon}_H-\alpha_{E^p}\mathscr{R}p_h
\qquad (\forall\,\,E^p\in \mathscr{T}_H^p)\\
\label{use2}
\bar{\sigma}=\eta_{E^f}\mathscr{P}\bar{\epsilon}_H-\alpha_{E^f}p_h
\qquad (\forall\,\,E^f\in \mathscr{T}_h^f)
\end{align}
\begin{theorem}\label{map}
In the presence of medium heterogeneities, the two-grid staggered solution algorithm in which the flow subproblem is resolved on a finer grid is a contraction map with contraction constant $\gamma$ and given by 
\begin{align}
\nonumber
&\sum\limits_{E^f\in \mathscr{T}_h^f} \frac{\Vert\delta^{(m)} \bar{\sigma}\Vert^2_{E^f}}{\eta_{E^f}}+\overbrace{\sum\limits_{E^p\in \mathscr{T}_H^p} 4G_{E^p}\Vert\boldsymbol{e}(\delta^{(m)} \mathbf{u}_H)\Vert^2_{E^p}}^{>0}+\overbrace{\sum\limits_{E^p\in \mathscr{T}_H^p} (2K_{b_{E^p}}-\eta_{E^p})\Vert\delta^{(m)} \bar{\epsilon}_H\Vert^2_{E^p}}^{>0}\\
\label{contraction}
&+\overbrace{\sum\limits_{E^f\in \mathscr{T}_h^f}
2\Delta t\Vert\boldsymbol{\kappa}_{E^f}^{-1/2}\delta^{(m)} \mathbf{z}_h \Vert^2_{E^f}}^{>0}\leq \overbrace{\max\limits_{E^f\in \mathscr{T}_h^f}\Bigg(\frac{\alpha_{E^f}^2}{\frac{\eta_{E^f}}{M_{E^f}}+\alpha_{E^f}^2}\Bigg)}^{\gamma<1}
\sum\limits_{E^f\in \mathscr{T}_h^f}
\frac{\Vert\delta^{(m-1)} \bar{\sigma}\Vert^2_{E^f}}{\eta_{E^f}}
\end{align}
if the following conditions are satisfied
\begin{enumerate}
\item First condition
\begin{align}
\nonumber
&\sum\limits_{E^f\in \mathscr{T}_h^f} \alpha_{E^f} (\mathscr{P}\delta^{(m)}
\bar{\epsilon}_H,\delta^{(m)} p_h)_{E^f}-\sum\limits_{E^p\in \mathscr{T}_H^p}\alpha_{E^p}(\delta^{(m)} \bar{\epsilon}_H ,\mathscr{R}\delta^{(m)} p_h)_{E^p}=0
\end{align}
\item Second condition
\begin{align}
\nonumber
&\sum\limits_{E^p\in \mathscr{T}_H^p} \eta_{E^p}\Vert\delta^{(m)} \bar{\epsilon}_H\Vert^2_{E^p}-\sum\limits_{E^f\in \mathscr{T}_h^f}\eta_{E^f}\Vert\mathscr{P}\delta^{(m)}
\bar{\epsilon}_H\Vert^2_{E^f}\geq 0
\end{align}
\item Third condition
\begin{align}
\eta_{E^p}\leq 2K_{b_{E^p}} \qquad (\forall E^p\in \mathscr{T}_H^p)
\end{align}
\end{enumerate}
\end{theorem}
\begin{proof}
$\bullet$ \textbf{Step 1: Flow equations}\\ 
Testing \eqref{msone} with $\theta_h\in W_h$ such that $\theta_h\vert_{E^f}= \delta^{(m)} p_h \,\,\forall\,\,E^f\in \mathscr{T}_h^f$, we get
\begin{align}
\label{msfour}
&\sum\limits_{E^f\in \mathscr{T}_h^f}\varphi_{E^f} \Vert \delta^{(m)} p_h \Vert^2_{E^f}+\sum\limits_{E^f\in \mathscr{T}_h^f}\Delta t(\nabla \cdot \delta^{(m)} \mathbf{z}_h,\delta^{(m)} p_h)_{E^f}=-\sum\limits_{E^f\in \mathscr{T}_h^f}\frac{\alpha_{E^f}}{\eta_{E^f}}(\delta^{(m-1)} \bar{\sigma},\delta^{(m)} p_h)_{E^f}
\end{align}
Testing \eqref{mstwo} with $\mathbf{v}_h\in \mathbf{V}_h$ such that $\mathbf{v}_h\vert_{E^f}\equiv \delta^{(m)} \mathbf{z}_h\,\,\forall\,\,E^f\in \mathscr{T}_h^f$, we get
\begin{align}
\label{msfive}
&\sum\limits_{E^f\in \mathscr{T}_h^f}
\Vert\boldsymbol{\kappa}_{E^f}^{-1/2}\delta^{(m)} \mathbf{z}_h \Vert^2_{E^f}=\sum\limits_{E^f\in \mathscr{T}_h^f}(\delta^{(m)} p_h,\nabla \cdot \delta^{(m)} \mathbf{z}_h)_{E^f}
\end{align}
From \eqref{msfour} and \eqref{msfive}, we get
\begin{align}
\label{mssix}
&\sum\limits_{E^f\in \mathscr{T}_h^f}\varphi_{E^f}\Vert \delta^{(m)} p_h \Vert^2+\sum\limits_{E^f\in \mathscr{T}_h^f}\Delta t
\Vert\boldsymbol{\kappa}_{E^f}^{-1/2}\delta^{(m)} \mathbf{z}_h \Vert^2_{E^f}=-\sum\limits_{E^f\in \mathscr{T}_h^f}\frac{\alpha_{E^f}}{\eta_{E^f}}(\delta^{(m-1)} \bar{\sigma},\delta^{(m)} p_h)_{E^f}
\end{align}
$\bullet$ \textbf{Step 2: Invoking the Young's inequality}\\
Since the terms on the LHS of \eqref{mssix} are strictly positive, the RHS is also strictly positive. We invoke the Young's inequality 
\begin{align*}
\vert ab\vert \leq \frac{a^2}{2\varepsilon}+\frac{\varepsilon b^2}{2}\qquad \forall\,\,a,b,\varepsilon\in \mathbb{R},\varepsilon> 0
\end{align*} 
for the RHS of \eqref{mssix} as follows
\begin{align}
\nonumber
&-\frac{\alpha_{E^f}}{\eta_{E^f}}(\delta^{(m-1)} \bar{\sigma},\delta^{(m)} p_h)_{E^f}\leq \frac{1}{2\varepsilon_{E^f}\eta_{E^f}^2}\Vert\delta^{(m-1)} \bar{\sigma}\Vert^2_{E^f}
+\frac{\varepsilon_{E^f}}{2}\Vert \alpha_{E^f} \delta^{(m)} p_h \Vert^2_{E^f} \,\,(\forall\,\,E^f \in \mathscr{T}_h^f)
\end{align}
Since the above inequality is true for any $\varepsilon_{E^f}> 0$, we choose $\varepsilon_{E^f}=\frac{1}{\alpha_{E^f}^2}\varphi_{E^f}$ to get
\begin{align}
\nonumber
&-\frac{\alpha_{E^f}}{\eta_{E^f}}(\delta^{(m-1)} \bar{\sigma},\delta^{(m)} p_h)_{E^f}\leq \frac{\alpha_{E^f}^2}{2\eta_{E^f}\varphi_{E^f}}\frac{\Vert\delta^{(m-1)} \bar{\sigma}\Vert^2_{E^f}}{\eta_{E^f}}+\frac{\varphi_{E^f}}{2}\Vert \delta^{(m)} p_h \Vert^2_{E^f}
\,\, (\forall\,\,E^f \in \mathscr{T}_h^f)
\end{align}
In lieu of the above, \eqref{mssix} is written as
\begin{align}
\nonumber
&\sum\limits_{E^f\in \mathscr{T}_h^f}\varphi_{E^f}\Vert \delta^{(m)} p_h \Vert^2_{E^f}+\sum\limits_{E^f\in \mathscr{T}_h^f}\Delta t
\Vert\boldsymbol{\kappa}_{E^f}^{-1/2}\delta^{(m)} \mathbf{z}_h \Vert^2_{E^f}\\
\nonumber
&\leq \sum\limits_{E^f\in \mathscr{T}_h^f}\frac{\alpha_{E^f}^2}{2\eta_{E^f}\varphi_{E^f}}\frac{\Vert\delta^{(m-1)} \bar{\sigma}\Vert^2_{E^f}}{\eta_{E^f}}
+\sum\limits_{E^f\in \mathscr{T}_h^f}\frac{\varphi_{E^f}}{2}\Vert \delta^{(m)} p_h \Vert^2_{E^f}  
\end{align}
which can also be written as
\begin{align}
\nonumber
&\sum\limits_{E^f\in \mathscr{T}_h^f}\frac{\varphi_{E^f}}{2} \Vert \delta^{(m)} p_h \Vert^2+\sum\limits_{E^f\in \mathscr{T}_h^f}\Delta t
\Vert\boldsymbol{\kappa}_{E^f}^{-1/2}\delta^{(m)} \mathbf{z}_h \Vert^2_{E^f}\leq \sum\limits_{E^f\in \mathscr{T}_h^f}\frac{\alpha_{E^f}^2}{2\eta_{E^f}\varphi_{E^f}}\frac{\Vert\delta^{(m-1)} \bar{\sigma}\Vert^2_{E^f}}{\eta_{E^f}}
\end{align}
which, after noting that $\varphi_{E^f}\equiv \bigg(\frac{1}{M_{E^f}}+\frac{\alpha_{E^f}^2}{\eta_{E^f}}\bigg)>\frac{\alpha_{E^f}^2}{\eta_{E^f}}$, can also be written as
\begin{align}
\label{interim3}
&\sum\limits_{E^f\in \mathscr{T}_h^f}\frac{\alpha_{E^f}^2}{\eta_{E^f}}\Vert \delta^{(m)} p_h \Vert^2_{E^f}+\sum\limits_{E^f\in \mathscr{T}_h^f}2\Delta t
\Vert\boldsymbol{\kappa}_{E^f}^{-1/2}\delta^{(m)} \mathbf{z}_h \Vert^2_{E^f}\leq \sum\limits_{E^f\in \mathscr{T}_h^f}\frac{\alpha_{E^f}^2}{\eta_{E^f}\varphi_{E^f}}\frac{\Vert\delta^{(m-1)} \bar{\sigma}\Vert^2_{E^f}}{\eta_{E^f}}
\end{align}
$\bullet$ \textbf{Step 3: Poromechanics equations}\\
Testing \eqref{msthree} with $\mathbf{q}_H\in \mathbf{Q}_H$ such that $\mathbf{q}\vert_{E^p}= 2\delta^{(m)}\mathbf{u}_H\,\,\forall\,\,E^p\in \mathscr{T}_H^p$
and noting that $\nabla \cdot \delta^{(m)} \mathbf{u}_H\equiv \delta^{(m)} \bar{\epsilon}_H$, we get
\begin{align}
\label{msrandom}
\sum\limits_{E^p\in \mathscr{T}_H^p} 4G_{E^p}\Vert\boldsymbol{e}(\delta^{(m)} \mathbf{u}_H)\Vert^2_{E^p}+\sum\limits_{E^p\in \mathscr{T}_H^p} 2(\delta^{(m)} \sigma_v,\delta^{(m)}\bar{\epsilon}_H)_{E^p}=0
\end{align}
Further, from \eqref{use1}, we note that $\delta^{(m)} \bar{\sigma}=K_{b_{E^p}}\delta^{(m)} \bar{\epsilon}_H-\alpha_{E^p} \mathscr{R}\delta^{(m)} p_h\,\,\forall\,\,E^p\in \mathscr{T}_H^p$. As a result, \eqref{msrandom} is written as
\begin{align}
\nonumber
&\sum\limits_{E^p\in \mathscr{T}_H^p} 4G_{E^p}\Vert\boldsymbol{e}(\delta^{(m)} \mathbf{u}_H)\Vert^2_{E^p}+\sum\limits_{E^p\in \mathscr{T}_H^p} 2K_{b_{E^p}}\Vert\delta^{(m)} \bar{\epsilon}_H\Vert^2_{E^p}\\
\label{msten}
&-\sum\limits_{E^p\in \mathscr{T}_H^p}2\alpha_{E^p} (\delta^{(m)}\bar{\epsilon}_H,\mathscr{R}\delta^{(m)} p_h)_{E^p}=0
\end{align}
$\bullet$ \textbf{Step 4: Combining flow and poromechanics equations}\\
Adding \eqref{interim3} and \eqref{msten}, we get
\begin{align}
\nonumber
&\sum\limits_{E^f\in \mathscr{T}_h^f}\frac{\alpha_{E^f}^2}{\eta_{E^f}}\Vert \delta^{(m)} p_h \Vert^2_{E^f}+\sum\limits_{E^f\in \mathscr{T}_h^f}2\Delta t
\Vert\boldsymbol{\kappa}_{E^f}^{-1/2}\delta^{(m)} \mathbf{z}_h \Vert^2_{E^f}+\sum\limits_{E^p\in \mathscr{T}_H^p} 4G_{E^p}\Vert\boldsymbol{e}(\delta^{(m)} \mathbf{u}_H)\Vert^2_{E^p}\\
\label{interim4}
&+\sum\limits_{E^p\in \mathscr{T}_H^p} 2K_{b_{E^p}}\Vert\delta^{(m)} \bar{\epsilon}_H\Vert^2_{E^p}-\sum\limits_{E^p\in \mathscr{T}_H^p}2\alpha_{E^p} (\delta^{(m)}\bar{\epsilon}_H,\mathscr{R}\delta^{(m)} p_h)_{E^p} \leq \sum\limits_{E^f\in \mathscr{T}_h^f}\frac{\alpha_{E^f}^2}{\eta_{E^f}\varphi_{E^f}}\frac{\Vert\delta^{(m-1)} \bar{\sigma}\Vert^2_{E^f}}{\eta_{E^f}}
\end{align}
Now, from \eqref{use2}, we note that 
\begin{align*}
&\Vert\delta^{(m)} \bar{\sigma}\Vert^2_{E^f}=\alpha_{E^f}^2\Vert \delta^{(m)} p_h \Vert^2_{E^f}+\eta_{E^f}^2\Vert\mathscr{P}
\delta^{(m)}\bar{\epsilon}_H\Vert^2_{E^f}-2\eta_{E^f} \alpha_{E^f}(\mathscr{P}\delta^{(m)}
\bar{\epsilon}_H,\delta^{(m)} p_h)_{E^f}\\
&(\forall\,\,E^f\in \mathscr{T}_h^f) 
\end{align*}
which implies that
\begin{align}
\nonumber
&\frac{\alpha_{E^f}^2}{\eta_{E^f}}\Vert \delta^{(m)} p_h \Vert^2_{E^f}=\frac{\Vert\delta^{(m)} \bar{\sigma}\Vert^2_{E^f}}{\eta_{E^f}}-\eta_{E^f}
\Vert\mathscr{P}\delta^{(m)}
\bar{\epsilon}_H\Vert^2_{E^f}
+2 \alpha_{E^f}(\mathscr{P}\delta^{(m)}
\bar{\epsilon}_H,\delta^{(m)} p_h)_{E^f}\\
\label{interim5}
&(\forall\,\,E^f\in \mathscr{T}_h^f) 
\end{align}
Substituting \eqref{interim5} in \eqref{interim4}, we get
\begin{align}
\nonumber
&\sum\limits_{E^f\in \mathscr{T}_h^f} \frac{\Vert\delta^{(m)} \bar{\sigma}\Vert^2_{E^f}}{\eta_{E^f}}+\overbrace{\sum\limits_{E^p\in \mathscr{T}_H^p} 4G_{E^p}\Vert\boldsymbol{e}(\delta^{(m)} \mathbf{u}_H)\Vert^2_{E^p}}^{>0}+\overbrace{\sum\limits_{E^p\in \mathscr{T}_H^p} (2K_{b_{E^p}}-\eta_{E^p})\Vert\delta^{(m)} \bar{\epsilon}_H\Vert^2_{E^p}}^{>0}\\
\nonumber
&+\Bigg[\overbrace{\sum\limits_{E^f\in \mathscr{T}_h^f} 2\alpha_{E^f}(\mathscr{P}\delta^{(m)}
\bar{\epsilon}_H,\delta^{(m)} p_h)_{E^f}-\sum\limits_{E^p\in \mathscr{T}_H^p}2\alpha_{E^p}(\delta^{(m)} \bar{\epsilon}_H, \mathscr{R}\delta^{(m)} p_h)_{E^p}}
^{\mathrm{Set\,\,=\,\,0\,\,to\,\,obtain\,\,expressions\,\,for\,\,\eta_{E^p}\,\,\forall\,\,E^p\in \mathscr{T}_H^p\,\,and\,\,\mathscr{P}\delta^{(m)}
\bar{\epsilon}_H\,\,\forall\,\,E^f\in \mathscr{T}_h^f}}\Bigg]\\
\nonumber
&+\Bigg[\overbrace{\sum\limits_{E^p\in \mathscr{T}_H^p} \eta_{E^p}\Vert \delta^{(m)} \bar{\epsilon}_H \Vert^2_{E^p}-\sum\limits_{E^f\in \mathscr{T}_h^f}\eta_{E^f}
\Vert\mathscr{P}\delta^{(m)}\bar{\epsilon}_H\Vert^2_{E^f}}
^{\mathrm{Turns\,\,out\,\,to\,\,be\,\,\geq 0\,\,in\,\,lieu\,\,of\,\,Cauchy-Schwartz\,\,inequality}}\Bigg]\\
\label{mseleven}
&+\overbrace{\sum\limits_{E^f\in \mathscr{T}_h^f}2\Delta t
\Vert\boldsymbol{\kappa}_{E^f}^{-1/2}\delta^{(m)} \mathbf{z}_h \Vert^2_{E^f}}^{>0}\leq \gamma\sum\limits_{E^f\in \mathscr{T}_h^f} {\frac{\Vert\delta^{(m-1)} \bar{\sigma}\Vert^2_{E^f}}{\eta_{E^f}}}
\end{align}
The statement \eqref{mseleven} is a contraction map in a sense that
\begin{align*}
\sum\limits_{E^f\in \mathscr{T}_h^f} \frac{\Vert\delta^{(0)} \bar{\sigma}\Vert^2_{E^f}}{\eta_{E^f}}>\sum\limits_{E^f\in \mathscr{T}_h^f} \frac{\Vert\delta^{(1)} \bar{\sigma}\Vert^2_{E^f}}{\eta_{E^f}}>\sum\limits_{E^f\in \mathscr{T}_h^f} \frac{\Vert\delta^{(2)} \bar{\sigma}\Vert^2_{E^f}}{\eta_{E^f}}>...
\end{align*}
with contraction constant $\gamma$ given by  
\begin{align}
\nonumber
\gamma \equiv \max\limits_{E^f\in \mathscr{T}_h^f}\bigg(\frac{\alpha_{E^f}^2}{\eta_{E^f}\varphi_{E^f}}\bigg)=\max\limits_{E^f\in \mathscr{T}_h^f}\bigg(\frac{\alpha_{E^f}^2}{\frac{\eta_{E^f}}{M_{E^f}}+\alpha_{E^f}^2}\bigg)<1
\end{align}
provided the following are true
\begin{align}
\label{condition1}
&\sum\limits_{E^f\in \mathscr{T}_h^f}\alpha_{E^f}(\mathscr{P}\delta^{(m)}\bar{\epsilon}_H,\delta^{(m)} p_h)_{E^f}-\sum\limits_{E^p\in \mathscr{T}_H^p}\alpha_{E^p}(\delta^{(m)} \bar{\epsilon}_H,\mathscr{R}\delta^{(m)} p_h)_{E^p}=0\\
\label{condition2}
&\sum\limits_{E^p\in \mathscr{T}_H^p} \eta_{E^p}\Vert \delta^{(m)} \bar{\epsilon}_H \Vert^2_{E^p}
-\sum\limits_{E^f\in \mathscr{T}_h^f}
\eta_{E^f}
\Vert\mathscr{P}\delta^{(m)}\bar{\epsilon}_H\Vert^2_{E^f}\geq 0\\
\label{condition3}
&\eta_{E^p}\leq 2K_{b_{E^p}} \qquad (\forall E^p\in \mathscr{T}_H^p)
\end{align}
\end{proof}
The objective now is to satisfy the conditions \eqref{condition1} and \eqref{condition2} for the convergence of the two-grid staggered solution algorithm.
\section{Satisfaction of conditions for the convergence of the fully discrete two-grid staggered solution algorithm}
\begin{corollary}
Satisfaction of the decoupling constraint during the flow solve at both scales leads to the following expressions for the upscaled pore pressures
\begin{align*}
\tag{$\forall\,\,E^p\in \mathcal{T}_H^p$}
&\mathscr{R}\delta^{(m)}p_h
=\frac{\eta_{E^p}}{\alpha_{E^p}}\sum\limits_{E^f\in \mathcal{I}^{E^p}}\frac{\alpha_{E^f}}{\eta_{E^f}}\delta^{(m)}p_h
\frac{Meas(E^f)}{Meas(E^p)}
\end{align*}
\end{corollary}
\begin{proof}
$\bullet$ \textbf{Step 1: Using the fact that pore pressure is frozen during the poromechanical solve}\\
Since the pore pressure is frozen during the poromechanical solve, the total pore pressure change in a coupling iteration is the same as the pore pressure change calculated during the flow solve in the coupling iteration as follows
\begin{align}
\label{use3}
&\mathscr{R}\delta_f^{(m)} p_h = \mathscr{R}\delta^{(m)} p_h\qquad(\forall\,\,E^p\in \mathscr{T}_H^p)\\
\label{use4}
&\delta_f^{(m)} p_h = \delta^{(m)} p_h\qquad(\forall\,\,E^f\in \mathscr{T}_h^f)
\end{align}
$\bullet$ \textbf{Step 2: Applying the decoupling constraint on both scales}\\
Now, the decoupling constraint implies that there is no change in the measure of the mean stress of the system during the flow solve. This naturally implies that
\begin{align}
\tag{$\forall\,\,E^p\in \mathscr{T}_H^p$}
\int\limits_{E^p} \delta_f^{(m)}\bar{\sigma}=0
\end{align}
In lieu of \eqref{use1}, we write the above as
\begin{align}
\tag{$\forall\,\,E^p\in \mathscr{T}_H^p$}
\int\limits_{E^p} (\eta_{E^p}\delta_f^{(m)}\bar{\epsilon}_H-\alpha_{E^p} \mathscr{R}\delta_f^{(m)}p_h) = 0
\end{align}
which, in lieu of \eqref{use3}, can be written as
\begin{align}
\label{feku}
\int\limits_{E^p}\delta_f^{(m)}\bar{\epsilon}_H = \frac{\alpha_{E^p}}{\eta_{E^p}}\mathscr{R}\delta^{(m)}p_h Meas(E^p)\qquad(\forall\,\,E^p\in \mathscr{T}_H^p)
\end{align}
Denoting $\int\limits_{E^f}
\delta_f^{(m)}\mathscr{P}\bar{\epsilon}_H$ is the change in volume of each element $E^f$ of $\mathcal{I}^{E^p}$, we now impose the decoupling constraint on each element $E^f$ of $\mathcal{I}^{E^p}$ as follows
\begin{align}
\nonumber
\int\limits_{E^f} \delta_f^{(m)}\bar{\sigma}\equiv \int\limits_{E^f} (\eta_{E^f}\mathscr{P}\delta_f^{(m)}\bar{\epsilon}_H-\alpha_{E^f} \delta^{(m)}p_h)=0\qquad(\forall\,\,E^f\in \mathcal{I}^{E^p})
\end{align}
which, in lieu of \eqref{use4}, can be written as 
\begin{align}
\label{feku2}
\int\limits_{E^f}
\mathscr{P}\delta_f^{(m)}\bar{\epsilon}_H=\frac{\alpha_{E^f}}{\eta_{E^f}}\delta^{(m)}p_hMeas(E^f)\qquad(\forall\,\,E^f\in \mathscr{T}_h^f)
\end{align} 
$\bullet$ \textbf{Step 3: Using the fact that the change in volume measured on both scales should be identical}\\
The term $\int\limits_{E^p}\delta_f^{(m)}\bar{\epsilon}_H$ is the change in volume of $E^p$ during the flow solve in the $(m+1)^{th}$ coupling iteration. This naturally equates the sum of corresponding changes in volumes of the elements of $\mathcal{I}^{E^p}$ as follows
\begin{align}
\label{feku1}
\int\limits_{E^p}\delta_f^{(m)}\bar{\epsilon}_H \equiv \sum\limits_{E^f\in \mathcal{I}^{E^p}}\int\limits_{E^f}
\mathscr{P}\delta_f^{(m)}\bar{\epsilon}_H\qquad(\forall\,\,E^p\in \mathscr{T}_H^p)
\end{align}
From \eqref{feku2} and \eqref{feku1}, we get
\begin{align}
\label{feku3}
\int\limits_{E^p}\delta_f^{(m)}\bar{\epsilon}_H = \sum\limits_{E^f\in \mathcal{I}^{E^p}}\frac{\alpha_{E^f}}{\eta_{E^f}}\delta^{(m)}p_h Meas(E^f)\qquad(\forall\,\,E^p\in \mathscr{T}_H^p)
\end{align}
From \eqref{feku} and \eqref{feku3}, we get
\begin{align*}
\sum\limits_{E^f\in \mathcal{I}^{E^p}}\frac{\alpha_{E^f}}{\eta_{E^f}}\delta^{(m)}p_h Meas(E^f)= \frac{\alpha_{E^p}}{\eta_{E^p}}\mathscr{R}\delta^{(m)}p_h Meas(E^p)\qquad(\forall\,\,E^p\in \mathscr{T}_H^p)
\end{align*}
which results in
\begin{align}
\label{feku4}
\mathscr{R}\delta^{(m)}p_h
=\frac{\eta_{E^p}}{\alpha_{E^p}}\sum\limits_{E^f\in \mathcal{I}^{E^p}}\frac{\alpha_{E^f}}{\eta_{E^f}}\delta^{(m)}p_h \frac{Meas(E^f)}{Meas(E^p)} \qquad(\forall\,\,E^p\in \mathscr{T}_H^p)
\end{align}
\end{proof}
\begin{corollary}
Satisfaction of the condition \eqref{condition1} leads to the following expressions for the effective bulk moduli for the coarse scale poromechanical solve  
\begin{align*}
\tag{$\forall\,\,E^p\in \mathscr{T}_H^p$}
&\eta_{E^p}=\frac{1}{\sum\limits_{E^f\in \mathcal{I}^{E^p}}\frac{1}{\eta_{E^f}}\frac{Meas(E^f)}{Meas(E^p)}}
\end{align*}
and the following expressions for the downscaled volumetric strains 
\begin{align*}
\tag{$\forall\,\,E^f\in \mathcal{I}^{E^p}\,\,\forall\,\,E^p\in \mathscr{T}_H^p$}
&\mathscr{P}\delta^{(m)}\bar{\epsilon}_H
=\frac{\eta_{E^p}}{\eta_{E^f}}\frac{1}{Meas(E^p)}\int\limits_{E^p}\delta^{(m)} \bar{\epsilon}_H
\end{align*}
\end{corollary}
\begin{proof}
$\bullet$ \textbf{Step 1: Recasting the first term on LHS of \eqref{condition1}}\\
We start by modifying the first term on LHS of \eqref{condition1} as follows
\begin{align}
\label{condition1r}
&\sum\limits_{E^f\in \mathscr{T}_h^f} \alpha_{E^f}(\mathscr{P}\delta^{(m)}\bar{\epsilon}_H,\delta^{(m)} p_h)_{E^f}= \sum\limits_{E^f\in \mathscr{T}_h^f}\alpha_{E^f}\delta^{(m)} p_h \mathscr{P}\delta^{(m)}\bar{\epsilon}_H Meas(E^f)
\end{align}
where we note that $\delta^{(m)}p_h\in W_h$. Since a flow element $E^f\in \mathscr{T}_h^f$ in uniquely associated with a poromechanical element $E^p$ via $\mathcal{I}^{E^p}$, we can write
\begin{align}
\nonumber
&\sum\limits_{E^f\in \mathscr{T}_h^f}\alpha_{E^f}\delta^{(m)} p_h\mathscr{P}\delta^{(m)}\bar{\epsilon}_H Meas(E^f)=\sum\limits_{E^p\in \mathscr{T}_H^p} \sum\limits_{E^f\in \mathcal{I}^{E^p}}\alpha_{E^f}\delta^{(m)} p_h\mathscr{P}\delta^{(m)}\bar{\epsilon}_H Meas(E^f)
\end{align}
In lieu of the above, we write \eqref{condition1r} as
\begin{align}
\label{condition1m}
&\sum\limits_{E^f\in \mathscr{T}_h^f} \alpha_{E^f}(\mathscr{P}\delta^{(m)}\bar{\epsilon}_H,\delta^{(m)} p_h)_{E^f}=\sum\limits_{E^p\in \mathscr{T}_H^p} \sum\limits_{E^f\in \mathcal{I}^{E^p}} \alpha_{E^f}\delta^{(m)} p_h\mathscr{P}\delta^{(m)}\bar{\epsilon}_H Meas(E^f)
\end{align}
$\bullet$ \textbf{Step 2: Recasting the second term on LHS of \eqref{condition1}}\\
Next, we modify the second term on LHS of \eqref{condition1} as follows
\begin{align}
\label{condition1mm}
&\sum\limits_{E^p\in \mathscr{T}_H^p}\alpha_{E^p}(\delta^{(m)} \bar{\epsilon}_H,\mathscr{R}\delta^{(m)} p_h)_{E^p}\equiv \sum\limits_{E^p\in \mathscr{T}_H^p}2\alpha_{E^p} \mathscr{R}\delta^{(m)} p_h\int\limits_{E^p}\delta^{(m)} \bar{\epsilon}_H
\end{align}
where we note that $\mathscr{R}\delta^{(m)}p_h\in W_H$. 
In lieu of \eqref{condition1m} and \eqref{condition1mm}, the first condition given by \eqref{condition1} is rewritten as
\begin{align}
\label{condition1l}
&\sum\limits_{E^p\in \mathscr{T}_H^p} \sum\limits_{E^f\in \mathcal{I}^{E^p}}\alpha_{E^f}\delta^{(m)} p_h\mathscr{P}\delta^{(m)}\bar{\epsilon}_H Meas(E^f)=\sum\limits_{E^p\in \mathscr{T}_H^p}\alpha_{E^p} \mathscr{R}\delta^{(m)} p_h\int\limits_{E^p}\delta^{(m)} \bar{\epsilon}_H
\end{align}
\noindent $\bullet$ \textbf{Step 3: Substituting the expression for upscaled pore pressures}\\
Substituting the expression \eqref{feku4} for the upscaled pore pressure in \eqref{condition1l}, we get
\begin{align*}
\nonumber
&\sum\limits_{E^p\in \mathscr{T}_H^p} \sum\limits_{E^f\in \mathcal{I}^{E^p}}\alpha_{E^f}\delta^{(m)} p_h\mathscr{P}\delta^{(m)}\bar{\epsilon}_H Meas(E^f)\\
&=\sum\limits_{E^p\in \mathscr{T}_H^p}\alpha_{E^p}\overbrace{\frac{\eta_{E^p}}{\alpha_{E^p}}\sum\limits_{E^f\in \mathcal{I}^{E^p}}\frac{\alpha_{E^f}}{\eta_{E^f}}\delta^{(m)}p_h \frac{Meas(E^f)}{Meas(E^p)}}^{\mathscr{R}\delta^{(m)} p_h}\int\limits_{E^p}\delta^{(m)} \bar{\epsilon}_H
\end{align*}
which implies that
\begin{align}
\nonumber
&\sum\limits_{E^p\in \mathscr{T}_H^p} \sum\limits_{E^f\in \mathcal{I}^{E^p}} 
\bigg(\mathscr{P}\delta^{(m)}\bar{\epsilon}_H 
-\frac{\eta_{E^p}}{\eta_{E^f}}\frac{1}{Meas(E^p)}\int\limits_{E^p}\delta^{(m)} \bar{\epsilon}_H\bigg)
\alpha_{E^f}\delta^{(m)}p_h Meas(E^f)=0
\end{align}
which, in lieu of the linear independence of the basis $p_h\,\,(\forall\,\,E^f\in \mathscr{T}_h^f)$ of the pressure space on the fine scale flow grid, implies that
\begin{align}
\nonumber
&\mathscr{P}\delta^{(m)}\bar{\epsilon}_H
-\frac{\eta_{E^p}}{\eta_{E^f}}\frac{1}{Meas(E^p)}\int\limits_{E^p}\delta^{(m)} \bar{\epsilon}_H=0\qquad (\forall\,\,E^f\in \mathcal{I}^{E^p}\,\,\forall\,\,E^p\in \mathscr{T}_H^p)
\end{align}
implying that
\begin{align}
\label{satisfy1}
&\mathscr{P}\delta^{(m)}\bar{\epsilon}_H
=\frac{\eta_{E^p}}{\eta_{E^f}}\frac{1}{Meas(E^p)}\int\limits_{E^p}\delta^{(m)} \bar{\epsilon}_H\qquad (\forall\,\,E^f\in \mathcal{I}^{E^p}\,\,\forall\,\,E^p\in \mathscr{T}_H^p)
\end{align}
$\bullet$ \textbf{Step 4: Using the fact that the change in volume measured on both scales should be identical}\\
The change in volume of $E^p$ over the $(m+1)^{th}$ coupling iteration equates the sum of corresponding changes in volumes of the elements of $\mathcal{I}^{E^p}$ as follows
\begin{align}
\label{satisfy2}
\int\limits_{E^p}\delta^{(m)}\bar{\epsilon}_H 
=\sum\limits_{E^f\in \mathcal{I}^{E^p}}\int\limits_{E^f}
\mathscr{P}\delta^{(m)}\bar{\epsilon}_H=\sum\limits_{E^f\in \mathcal{I}^{E^p}}\mathscr{P}\delta^{(m)}\bar{\epsilon}_H Meas(E^f)\qquad (\forall\,\,E^p\in \mathscr{T}_H^p)
\end{align}
In lieu of \eqref{satisfy1} and \eqref{satisfy2}, we get
\begin{align}
\tag{$\forall\,\,E^p\in \mathscr{T}_H^p$}
\int\limits_{E^p}\delta^{(m)}\bar{\epsilon}_H 
= \int\limits_{E^p}\delta^{(m)} \bar{\epsilon}_H\sum\limits_{E^f\in \mathcal{I}^{E^p}}\frac{\eta_{E^p}}{\eta_{E^f}}\frac{Meas(E^f)}{Meas(E^p)}
\end{align}
which finally leads to
\begin{align}
\label{satisfy3}
\frac{1}{\eta_{E^p}}
= \sum\limits_{E^f\in \mathcal{I}^{E^p}}\frac{1}{\eta_{E^f}}\frac{Meas(E^f)}{Meas(E^p)}\qquad (\forall\,\,E^p\in \mathscr{T}_H^p)
\end{align}
\end{proof}
\begin{corollary}
The Cauchy-Schwartz inequality, along with the obtained expressions for effective coarse scale bulk moduli \eqref{satisfy3} and downscaled volumetric strains \eqref{satisfy1}, guarantees the satisfaction of the condition \eqref{condition2}.
\end{corollary}
\begin{proof}
$\bullet$ \textbf{Step 1: Recasting \eqref{condition2} in lieu of \eqref{satisfy1} and \eqref{satisfy3}}\\
The condition \eqref{condition2} given by
\begin{align*}
&\sum\limits_{E^p\in \mathscr{T}_H^p} \eta_{E^p}\Vert \delta^{(m)} \bar{\epsilon}_H \Vert^2_{E^p}
-\sum\limits_{E^f\in \mathscr{T}_h^f}
\eta_{E^f}
\Vert\mathscr{P}\delta^{(m)}\bar{\epsilon}_H\Vert^2_{E^f}\geq 0
\end{align*}
can be written as
\begin{align*}
&\sum\limits_{E^p\in \mathscr{T}_H^p} \eta_{E^p}\Vert \delta^{(m)} \bar{\epsilon}_H \Vert^2_{E^p}
-\sum\limits_{E^f\in \mathscr{T}_H^f} 
\eta_{E^f}
\vert\mathscr{P}\delta^{(m)}\bar{\epsilon}_H\vert^2 Meas(E^f)\geq 0
\end{align*}
which can also be written as
\begin{align*}
&\sum\limits_{E^p\in \mathscr{T}_H^p}\bigg[ \eta_{E^p}\Vert \delta^{(m)} \bar{\epsilon}_H \Vert^2_{E^p}
-\sum\limits_{E^f\in \mathcal{I}^{E^p}} 
\eta_{E^f}
\vert\mathscr{P}\delta^{(m)}\bar{\epsilon}_H\vert^2 Meas(E^f)\bigg]\geq 0
\end{align*}
which, in lieu of \eqref{satisfy1}, can also be written as
\begin{align*}
&\sum\limits_{E^p\in \mathscr{T}_H^p}\bigg[ \eta_{E^p}\Vert \delta^{(m)} \bar{\epsilon}_H \Vert^2_{E^p}
-\sum\limits_{E^f\in \mathcal{I}^{E^p}} 
\eta_{E^f}
\overbrace{\bigg(\frac{\eta_{E^p}}{\eta_{E^f}}\frac{1}{Meas(E^p)}\int\limits_{E^p}\delta^{(m)} \bar{\epsilon}_H\bigg)^2}^{\vert\mathscr{P}\delta^{(m)}\bar{\epsilon}_H\vert^2} Meas(E^f)\bigg]\geq 0
\end{align*}
which can also be written as
\begin{align*}
&\sum\limits_{E^p\in \mathscr{T}_H^p}\bigg[ \eta_{E^p}\Vert \delta^{(m)} \bar{\epsilon}_H \Vert^2_{E^p}
-\frac{\eta_{E^p}}{Meas(E^p)}\bigg(\int\limits_{E^p} \vert \delta^{(m)}\bar{\epsilon}_H\vert\bigg)^2
\eta_{E^p}\sum\limits_{E^f\in \mathcal{I}^{E^p}}\frac{1}{\eta_{E^f}}\frac{Meas(E^f)}{Meas(E^p)}
\bigg]\geq 0
\end{align*}
which, in lieu of \eqref{satisfy3}, can be written as
\begin{align*}
&\sum\limits_{E^p\in \mathscr{T}_H^p}\bigg[ \eta_{E^p}\Vert \delta^{(m)} \bar{\epsilon}_H \Vert^2_{E^p}
-\frac{\eta_{E^p}}{Meas(E^p)}\bigg(\int\limits_{E^p} \vert \delta^{(m)}\bar{\epsilon}_H\vert\bigg)^2
\eta_{E^p}\frac{1}{\eta_{E^p}}
\bigg]\geq 0
\end{align*}
which can be finally written as
\begin{align}
\label{satisfy4}
&\sum\limits_{E^p\in \mathscr{T}_H^p}\bigg[ \eta_{E^p}\Vert \delta^{(m)} \bar{\epsilon}_H \Vert^2_{E^p}
-\frac{\eta_{E^p}}{Meas(E^p)}\bigg(\int\limits_{E^p} \vert \delta^{(m)}\bar{\epsilon}_H\vert\bigg)^2
\bigg]\geq 0
\end{align}
$\bullet$ \textbf{Step 2: Applying the Cauchy-Schwartz inequality}\\
The Cauchy-Schwartz inequality (see \citet{functionaloden}) states that if $S$ is a measurable subset of $\mathbb{R}^3$ and $f$ and $g$ are measurable real-valued or complex-valued functions on $S$, then the following is true
\begin{align*}
\bigg(\int\limits_S \vert fg\vert\bigg)^2\leq  \Vert f\Vert^2_S \Vert g\Vert^2_S
\end{align*}
Replacing $S$ by $E^p$, $f$ by $\delta^{(m)}\bar{\epsilon}_H$ and $g$ by $1$, we get
\begin{align*}
\tag{$\forall\,\,E^p\in \mathscr{T}_H^p$}
\Vert \delta^{(m)}\bar{\epsilon}_H\Vert^2_{E^p} \geq \frac{1}{Meas(E^p)}\bigg(\int\limits_{E^p} \vert \delta^{(m)}\bar{\epsilon}_H\vert\bigg)^2
\end{align*}
which can be written as
\begin{align*}
\tag{$\forall\,\,E^p\in \mathscr{T}_H^p$}
\eta_{E^p}\Vert \delta^{(m)} \bar{\epsilon}_H \Vert^2_{E^p}
-\frac{\eta_{E^p}}{Meas(E^p)}\bigg(\int\limits_{E^p} \vert \delta^{(m)}\bar{\epsilon}_H\vert\bigg)^2\geq 0
\end{align*}
which implies that
\begin{align}
\nonumber
&\sum\limits_{E^p\in \mathscr{T}_H^p}\bigg[ \eta_{E^p}\Vert \delta^{(m)} \bar{\epsilon}_H \Vert^2_{E^p}
-\frac{\eta_{E^p}}{Meas(E^p)}\bigg(\int\limits_{E^p} \vert \delta^{(m)}\bar{\epsilon}_H\vert\bigg)^2
\bigg]\geq 0
\end{align}
which is identical to \eqref{satisfy4}. Thus, provided the downscaled volumetric strains are computed in accordance with \eqref{satisfy1} and effective coarse scale bulk moduli are computed in accordance with \eqref{satisfy3}, the Cauchy-Schwartz inequality guarantees the satisfaction of the condition \eqref{condition2}.
\end{proof}
\section{The Voigt bound, the Reuss bound and the contraction constant}
The contraction constant is given by
\begin{align}
\nonumber
\gamma =\max\limits_{E^f\in \mathscr{T}_h^f}\bigg(\frac{\alpha_{E^f}^2}{\frac{\eta_{E^f}}{M_{E^f}}+\alpha_{E^f}^2}\bigg)<1
\end{align}
It is clear to see that the minimum value of contraction constant is obtained when the adjustable parameter takes the maximum possible value. To interrogate the maximum value that the adjustable parameter can achieve, we look at the third condition for the satisfaction of the contractivity given by 
\begin{align*}
&\eta_{E^p}\leq 2K_{b_{E^p}} \qquad (\forall E^p\in \mathscr{T}_H^p)
\end{align*}
It is clear when $\eta=2K_b$, we obtain the minimum contraction constant thus implying fastest convergence of the staggered solution algorithm. The expression \eqref{satisfy4} for the coarse scale moduli in terms on fine scale data is given by
\begin{align}
\label{randomass}
\frac{1}{\eta_{E^p}}
= \sum\limits_{E^f\in \mathcal{I}^{E^p}}\frac{1}{\eta_{E^f}}\frac{Meas(E^f)}{Meas(E^p)}\qquad (\forall\,\,E^p\in \mathscr{T}_H^p)
\end{align}
The following cases arise
\begin{itemize}
\item The adjustable parameter is equal to twice the drained bulk modulus i.e. $\eta\equiv 2K_b$
\begin{align}
\nonumber
&\frac{1}{2K_{b_{E^p}}}
= \sum\limits_{E^f\in \mathcal{I}^{E^p}}\frac{1}{2K_{b_{E^f}}}\frac{Meas(E^f)}{Meas(E^p)}\qquad (\forall\,\,E^p\in \mathscr{T}_H^p)\\
\nonumber
&\implies \frac{1}{K_{b_{E^p}}}
= \sum\limits_{E^f\in \mathcal{I}^{E^p}}\frac{1}{K_{b_{E^f}}}\frac{Meas(E^f)}{Meas(E^p)}\qquad (\forall\,\,E^p\in \mathscr{T}_H^p)
\end{align}
In this case, the coarse scale bulk moduli are harmonic mean of the fine scale data, thus representing the Reuss bound
\item The adjustable parameter is equal to inverse of the drained bulk modulus i.e. $\eta\equiv \frac{1}{K_b}$
\begin{align}
\nonumber
K_{b_{E^p}}
= \sum\limits_{E^f\in \mathcal{I}^{E^p}}K_{b_{E^f}}\frac{Meas(E^f)}{Meas(E^p)}\qquad (\forall\,\,E^p\in \mathscr{T}_H^p)
\end{align}
In this case, the coarse scale bulk moduli are arithmetic mean of the fine scale data, thus representing the Voigt bound
\end{itemize}
We already know that the Reuss and Voigt bounds on effective moduli yield the lower and upper bounds for the elastic strain energy for multiphase composites respectively (see \cite{saeb-2016}). In lieu of that, we state that the adjustable parameter is bounded above by the drained bulk modulus and below by the inverse of bulk modulus as follows
\begin{align*}
\frac{1}{K_b}\leq \eta \leq 2K_b
\end{align*}
\section{Conclusions and outlook}
The link we established between the measure of the mean stress used in the decoupling constraint and the Voigt and Reuss bounds has interesting connotations for the imposed homogeneous boundary conditions used to arrive at effective properties in the computational homogenization of multiphase composites. We know that stress uniform boundary conditions on the mesoscale lead to the Reuss bound on the effective property at the macroscale while the kinematic uniform boundary conditions on the mesoscale lead to the Voigt bound on the effective property at the macroscale (\citet{hashin-1962}, \citet{hill-1963}, \citet{hillmandel1}, \citet{hill-1972}, \citet{hashin-1983}, \citet{zohdi}). We also know that periodic boundary conditions on the mesoscale lead to the most accurate effective properties at the macroscale. In case of the two-grid approach, the fine scale flow grid is the mesoscale while the coarse scale poromechanical grid is the macroscale. When the adjustable parameter takes upon the value of twice the drained bulk modulus, we obtain the Reuss bound corresponding to stress uniform boundary conditions on the mesoscale. Similarly, when the adjustable parameter takes upon the value of the inverse of the drained bulk modulus, we obtain the Voigt bound corresponding to kinematic uniform boundary conditions on the mesoscale. By an extension of that logic, we expect a certain value of the adjustable parameter that corresponds to the periodic boundary conditions imposed on the mesoscale thereby lending itself to the most accurate estimate of the macroscale effective property, and thereby lending itself to the fastest convergence of the two-grid staggered solution algorithm.
\appendix
\section{Discrete variational statements for the flow subproblem in terms of coupling iteration differences}\label{discreteflow}
Before arriving at the discrete variational statement of the flow model, we impose the decoupling constraint on the strong form of the mass conservation equation \eqref{massagain1}. Invoking the relation $\bar{\sigma}=\eta \bar{\epsilon}-\alpha p$, we get
\begin{align}
\nonumber
&\frac{\partial}{\partial t}\bigg(\frac{1}{M}p+\alpha \bigg(\frac{\bar{\sigma}+\alpha p}{\eta}\bigg) \bigg)+\nabla \cdot \mathbf{z}=q\\
\label{ek}
&\overbrace{\bigg(\frac{1}{M}+\frac{\alpha^2}{\eta}\bigg)}^{\varphi}\frac{\partial p}{\partial t} +\nabla \cdot \mathbf{z}=q-\frac{\alpha}{\eta}\frac{\partial \bar{\sigma}}{\partial t}
\end{align}
Using backward Euler in time, the discrete in time form of \eqref{ek} for the $m^{th}$ coupling iteration in the $(n+1)^{th}$ time step is written as
\begin{align*}
&\varphi\frac{1}{\Delta t}(p^{m,n+1}-p^n) +\nabla \cdot \mathbf{z}^{n+1}=q^{n+1}-\frac{\alpha}{\eta}\frac{1}{\Delta t}(\bar{\sigma}^{m,n+1}-\bar{\sigma}^n)
\end{align*}
where $\Delta t$ is the time step and the source term as well as the terms evaluated at the previous time level $n$ do not depend on the coupling iteration count as they are known quantities. The decoupling constraint implies that $\bar{\sigma}^{m,n+1}$ gets replaced by $\bar{\sigma}^{m-1,n+1}$ i.e. the computation of $p^{m,n+1}$ and $\mathbf{z}^{m,n+1}$ is based on the value of $\bar{\sigma}$ updated after the poromechanics solve from the previous coupling iteration $m-1$ at the current time level $n+1$. The modified equation is written as
\begin{align}
\nonumber
&\varphi(p^{m,n+1}-p^n)+\Delta t\nabla \cdot \mathbf{z}^{m,n+1}=\Delta t q^{n+1}-\frac{\alpha}{\eta}(\bar{\sigma}^{m-1,n+1}-\bar{\sigma}^n)
\end{align}
As a result, the discrete variational statement of \eqref{massagain1} in the presence of medium heterogeneities is
\begin{align}
\nonumber
&\sum\limits_{E^f\in \mathscr{T}_h^f}\varphi_{E^f}(p_h^{m,n+1}-p_h^n,\theta_h)_{E^f}+\sum\limits_{E^f\in \mathscr{T}_h^f}\Delta t(\nabla \cdot \mathbf{z}_h^{m,n+1},\theta_h)_{E^f}\\
\label{wone}
&=\sum\limits_{E^f\in \mathscr{T}_h^f}\Delta t(q^{n+1},\theta_h)_{E^f}-\sum\limits_{E^f\in \mathscr{T}_h^f}\frac{\alpha_{E^f}}{\eta_{E^f}}
(\bar{\sigma}^{m-1,n+1}-\bar{\sigma}^n,\theta_h)_{E^f}
\end{align}
Replacing $m$ by $m+1$ in \eqref{wone} and subtracting the two equations, we get 
\begin{align*}
&\sum\limits_{E^f\in \mathscr{T}_h^f}\varphi_{E^f}(\delta^{(m)} p_h,\theta_h)_{E^f}+\sum\limits_{E^f\in \mathscr{T}_h^f}\Delta t(\nabla \cdot \delta^{(m)} \mathbf{z}_h,\theta_h)_{E^f}=-\sum\limits_{E^f\in \mathscr{T}_h^f}\frac{\alpha_{E^f}}{\eta_{E^f}}(\delta^{(m-1)} \bar{\sigma},\theta_h)_{E^f}
\end{align*}
The weak form of the Darcy law \eqref{darcyagain1} for the $m^{th}$ coupling iteration in the $(n+1)^{th}$ time step is
\begin{align}
\label{wtwo1}
(\boldsymbol{\kappa}^{-1}\mathbf{z}^{m,n+1},\mathbf{v})_{\Omega}=-(\nabla p^{m,n+1},\mathbf{v})_{\Omega}+(\rho_0 \mathbf{g},\mathbf{v})_{\Omega}\qquad \forall\,\,\mathbf{v}\in \mathbf{V}(\Omega)
\end{align}
where $\mathbf{V}(\Omega)$ is given by
\begin{align*}
\mathbf{V}(\Omega)\equiv \mathbf{H}(div,\Omega)\cap \big\{\mathbf{v}:\mathbf{v}\cdot \mathbf{n}=0\,\,\mathrm{on}\,\,\Gamma_N^f\big\}
\end{align*}
and $\mathbf{H}(div,\Omega)$ is given by 
\begin{align*}
\mathbf{H}(div,\Omega)\equiv\big\{\mathbf{v}:\mathbf{v}\in (L^2(\Omega))^3,\nabla \cdot \mathbf{v}\in L^2(\Omega) \big\}
\end{align*}
We use the divergence theorem to evaluate the first term on RHS of \eqref{wtwo1} as follows
\begin{align}
\nonumber
&(\nabla p^{m,n+1},\mathbf{v})_{\Omega}=(\nabla,p^{m,n+1}\mathbf{v})_{\Omega}-(p^{m,n+1},\nabla \cdot \mathbf{v})_{\Omega}\\
\label{wtwo2}
&=(p^{m,n+1},\mathbf{v}\cdot \mathbf{n})_{\partial \Omega} -(p^{m,n+1},\nabla \cdot \mathbf{v})_{\Omega}=(g,\mathbf{v}\cdot \mathbf{n})_{\Gamma_D^f} -(p^{m,n+1},\nabla \cdot \mathbf{v})_{\Omega}
\end{align}
where we invoke $\mathbf{v}\cdot \mathbf{n}=0$ on $\Gamma_N^f$.
In lieu of \eqref{wtwo1} and \eqref{wtwo2}, we get
\begin{align*}
(\boldsymbol{\kappa}^{-1}\mathbf{z}^{m,n+1},\mathbf{v})_{\Omega}=-(g,\mathbf{v}\cdot \mathbf{n})_{\Gamma_D^f}+(p^{m,n+1},\nabla \cdot \mathbf{v})_{\Omega}+(\rho_0 \mathbf{g},\mathbf{v})_{\Omega}
\end{align*}
As a result, the discrete variational statement of \eqref{darcyagain1} in the presence of medium heterogeneities is
\begin{align}
\nonumber
&\sum\limits_{E^f\in \mathscr{T}_h^f}
(\boldsymbol{\kappa}_{E^f}^{-1}\mathbf{z}_h^{m,n+1},\mathbf{v}_h)_{E^f}-\sum\limits_{E^f\in \mathscr{T}_h^f} (p_h^{m,n+1},\nabla \cdot \mathbf{v}_h)_{E^f}\\
\label{wtwo}
&=\sum\limits_{E^f\in \mathscr{T}_h^f}(\rho_0 \mathbf{g}, \mathbf{v}_h)_{E^f}-\sum\limits_{E^f\in \mathscr{T}_h^f}(g,\mathbf{v}_h\cdot \mathbf{n})_{\partial E^f\cap \Gamma_D^f}
\end{align}
Replacing $m$ by $m+1$ in \eqref{wtwo} and subtracting the two equations, we get
\begin{align}
\nonumber
&\sum\limits_{E^f\in \mathscr{T}_h^f}
(\boldsymbol{\kappa}_{E^f}^{-1}\delta^{(m)} \mathbf{z}_h, \mathbf{v}_h)_{E^f}=\sum\limits_{E^f\in \mathscr{T}_h^f}(\delta^{(m)} p_h,\nabla \cdot \mathbf{v}_h)_{E^f}
\end{align} 
\section{Discrete variational statement for the poromechanics subproblem in terms of coupling iteration differences}\label{discretemechanics}
The weak form of the linear momentum balance \eqref{mechstart} is given by
\begin{align}
\label{app1}
(\nabla \cdot \boldsymbol{\sigma},\mathbf{q})_{\Omega}+(\mathbf{f}\cdot \mathbf{q})_{\Omega}=0\qquad (\forall\,\,\mathbf{q}\in \mathbf{U}(\Omega))
\end{align}
where $\mathbf{U}(\Omega)$ is given by
\begin{align*}
\mathbf{U}(\Omega)\equiv \big\{\mathbf{q}=(u,v,w):u,v,w\in H^1(\Omega),\mathbf{q}=\mathbf{0}\,\,\mathrm{on}\,\,\Gamma_D^p\big\}
\end{align*}
where $H^m(\Omega)$ is defined, in general, for any integer $m\geq 0$ as
\begin{align*}
H^m(\Omega)\equiv\big\{w:D^{\alpha}w\in L^2(\Omega)\,\,\forall |\alpha| \leq m \big\},
\end{align*}
where the derivatives are taken in the sense of distributions and given by
\begin{align*}
D^{\alpha}w=\frac{\partial^{|\alpha|}w}{\partial x_1^{\alpha_1}..\partial x_n^{\alpha_n}},\,\,|\alpha|=\alpha_1+\cdots+\alpha_n,
\end{align*} 
We know from tensor calculus that
\begin{align}
\label{app2}
(\nabla \cdot \boldsymbol{\sigma},\mathbf{q})\equiv (\nabla ,\boldsymbol{\sigma}\mathbf{q})-(\boldsymbol{\sigma},\nabla \mathbf{q})
\end{align}
Further, using the divergence theorem and the symmetry of $\boldsymbol{\sigma}$, we arrive at
\begin{align}
\label{app3}
(\nabla ,\boldsymbol{\sigma}\mathbf{q})_{\Omega}\equiv (\mathbf{q},\boldsymbol{\sigma}\mathbf{n})_{\partial \Omega}
\end{align}
We decompose $\nabla \mathbf{q}$ into a symmetric part $(\nabla \mathbf{q})_{s}\equiv \frac{1}{2}\big(\nabla \mathbf{q}+(\nabla \mathbf{q})^T\big)\equiv \boldsymbol{\epsilon}(\mathbf{q})$ and skew-symmetric part $(\nabla \mathbf{q})_{ss}$ and note that the contraction between a symmetric and skew-symmetric tensor is zero to obtain
\begin{align}
\label{app4}
\boldsymbol{\sigma}:\nabla \mathbf{q}\equiv \boldsymbol{\sigma}:(\nabla \mathbf{q})_{s}+\cancelto{0}{\boldsymbol{\sigma}:(\nabla \mathbf{q})_{ss}}=\boldsymbol{\sigma}:\boldsymbol{\epsilon}(\mathbf{q})
\end{align}
From \eqref{app1}, \eqref{app2}, \eqref{app3} and \eqref{app4}, we get
\begin{align*}
&(\boldsymbol{\sigma}\mathbf{n},\mathbf{q})_{\partial \Omega} - (\boldsymbol{\sigma},\boldsymbol{\epsilon}(\mathbf{q}))_{\Omega} + (\mathbf{f},\mathbf{q})_{\Omega}=0
\end{align*}
which, after invoking the boundary condition $\boldsymbol{\sigma}\mathbf{n}=\mathbf{t}$ on $\Gamma_N^p$ results in
\begin{align}
\label{app5}
&(\mathbf{t},\mathbf{q})_{\Gamma_N^p} - (\boldsymbol{\sigma},\boldsymbol{\epsilon}(\mathbf{q}))_{\Omega} + (\mathbf{f},\mathbf{q})_{\Omega}=0
\end{align}
The stress tensor $\boldsymbol{\sigma}$ and strain tensor $\boldsymbol{\epsilon}(\mathbf{q})$ are written as
\begin{align*}
&\boldsymbol{\sigma}=\boldsymbol{s}+\frac{1}{3}tr(\boldsymbol{\sigma})\mathbf{I}=\boldsymbol{s}
+\sigma_v\mathbf{I};\qquad \boldsymbol{\epsilon}(\mathbf{q})=\boldsymbol{e}(\mathbf{q})+\frac{1}{3}tr(\boldsymbol{\epsilon}(\mathbf{q}))\mathbf{I}=\boldsymbol{e}(\mathbf{q})+\frac{1}{3}\bar{\epsilon}(\mathbf{q})\mathbf{I}
\end{align*}
where $\boldsymbol{s}$ is the deviatoric stress tensor, $\boldsymbol{e}(\mathbf{q})$ is the deviatoric strain tensor and $\sigma_v$ is the mean stress. Using the above relations, we can write
\begin{align}
\nonumber
&\boldsymbol{\sigma}:\boldsymbol{\epsilon}(\mathbf{q})=
\big(\boldsymbol{s}
+\sigma_v\mathbf{I}\big):\big(\boldsymbol{e}(\mathbf{q})+\frac{1}{3}\bar{\epsilon}(\mathbf{q})\mathbf{I}\big)=\boldsymbol{s}:\boldsymbol{e}(\mathbf{q})
+\boldsymbol{s}:\frac{1}{3}\bar{\epsilon}(\mathbf{q})\mathbf{I}
+\sigma_v\mathbf{I}:\boldsymbol{e}(\mathbf{q})
+\sigma_v\mathbf{I}:\frac{1}{3}\bar{\epsilon}(\mathbf{q})\mathbf{I}\\
\label{app7}
&=\boldsymbol{s}:\boldsymbol{e}(\mathbf{q})
+\frac{1}{3}\bar{\epsilon}(\mathbf{q})\cancelto{0}{tr(\boldsymbol{s})}
+\sigma_v \cancelto{0}{tr(\boldsymbol{e}(\mathbf{q}))}
+3\sigma_v\frac{1}{3}\bar{\epsilon}(\mathbf{q})=\boldsymbol{s}:\boldsymbol{e}(\mathbf{q})+\sigma_v\bar{\epsilon}(\mathbf{q})
\end{align}
where we note that the contraction of any second order tensor with the identity tensor $\mathbf{I}$ is equal to the trace of the tensor and further, the trace of a deviatoric tensor is zero resulting in $tr(\boldsymbol{s})=0$ and $tr(\boldsymbol{e}(\mathbf{q}))=0$.
Substituting \eqref{app7} in \eqref{app5}, we get
\begin{align}
\label{app8}
&(\mathbf{t},\mathbf{q})_{\Gamma_N^p}-(\mathbf{s},\mathbf{e}(\mathbf{q}))_{\Omega} - (\sigma_v,\bar{\epsilon}(\mathbf{q}))_{\Omega} + (\mathbf{f},\mathbf{q})_{\Omega}=0
\end{align}
The deviatoric strain tensor is obtained as
\begin{align}
\nonumber
\boldsymbol{s}&=\boldsymbol{\sigma}-\frac{1}{3}tr(\boldsymbol{\sigma})\mathbf{I}=\boldsymbol{\sigma}_0
+\lambda \bar{\epsilon}\mathbf{I}+2G\boldsymbol{\epsilon}
-\alpha (p-p_0)\mathbf{I}-\frac{1}{3}tr\big(\boldsymbol{\sigma}_0
+\lambda \bar{\epsilon}\mathbf{I}+2G\boldsymbol{\epsilon}
-\alpha (p-p_0)\mathbf{I}\big)\mathbf{I}\\
\label{app11}
&=\mathbf{s}_0+2G\big(\boldsymbol{\epsilon}-\frac{1}{3}tr(\boldsymbol{\epsilon})\mathbf{I}\big)
=\mathbf{s}_0+2G\boldsymbol{e}
\end{align}
Substituting \eqref{app11} in \eqref{app8}, we get
\begin{align}
\nonumber
&(\mathbf{t},\mathbf{q})_{\Gamma_N^p}-(\mathbf{s}_0,\mathbf{e}(\mathbf{q}))_{\Omega}-(2G\mathbf{e},\mathbf{e}(\mathbf{q}))_{\Omega} - (\sigma_v,\bar{\epsilon}(\mathbf{q}))_{\Omega} + (\mathbf{f},\mathbf{q})_{\Omega}=0
\end{align}
As a result, the discrete variational statement of the linear momentum balance \eqref{mechstart} for the $m^{th}$ coupling iteration in the $(n+1)^{th}$ time step in the presence of medium heterogeneities is written as 
\begin{align}
\nonumber
&\sum\limits_{E^p\in \mathscr{T}_H^p}2G_{E^p}(\boldsymbol{e}(\mathbf{u}_H^{m,n+1}),\boldsymbol{e}(\mathbf{q}_H))_{E^p}
+\sum\limits_{E^p\in \mathscr{T}_H^p}
(\sigma_v^{m,n+1},\nabla \cdot \mathbf{q}_H)_{E^p}\\
\label{wthree}
&=\sum\limits_{E^p\in \mathscr{T}_H^p}(\mathbf{f},\mathbf{q}_H)+\sum\limits_{E^p\in \mathscr{T}_H^p}(\mathbf{t},\mathbf{q}_H)_{\partial E^p\cap \Gamma_N^p}-\sum\limits_{E^p\in \mathscr{T}_H^p}(\boldsymbol{s}_0,\boldsymbol{e}(\mathbf{q}_H))_{E^p}
\end{align}
Replacing $m$ by $m+1$ in \eqref{wthree} and subtracting the two equations, we get
\begin{align}
\nonumber
&\sum\limits_{E^p\in \mathscr{T}_H^p} 2G_{E^p}(\boldsymbol{e}(\delta^{(m)} \mathbf{u}_H),\boldsymbol{e}(\mathbf{q}_H))_{E^p}+\sum\limits_{E^p\in \mathscr{T}_H^p}(\delta^{(m)} \sigma_v,\nabla \cdot \mathbf{q}_H)_{E^p}=0
\end{align}
\bibliographystyle{plainnat} 
\bibliography{diss}
\end{document}